\documentclass{amsart}
  
  \usepackage{amsthm} 
  \usepackage{amsmath} 
  \usepackage{amssymb} 
  \usepackage{mathtools} 

 \usepackage{graphicx}
  \usepackage{xspace}
  \usepackage[all]{xy}
  \usepackage{pinlabel}
  \usepackage{enumerate}
  
  \usepackage{bbm}
  \usepackage{xfrac}
  
  \usepackage[margin=1.3in]{geometry}
  \usepackage{tikz}
  \usetikzlibrary{calc,through,backgrounds, arrows}
  \usetikzlibrary{decorations.pathmorphing, decorations.pathreplacing,positioning,fit}

\newcommand{\cherry}{{%
\begin{tikzpicture}%

\tikzstyle{vertex} =[circle,draw,fill=black,thick, inner sep=0pt,minimum size= .2 mm]

\node [vertex] (x) at (0,0)  {};
\node [vertex] (y) at (-.06,-.1)  {};
\node [vertex] (z) at (0.06,-0.1)  {};

\draw (y) -- (x) -- (z);%
\end{tikzpicture}%
}}

\newcommand{\homsc}{{%
		\begin{tikzpicture}%
			
			\tikzstyle{vertex} =[circle,draw,fill=black,thick, inner sep=0pt,minimum size= .2 mm]
			
			\node [vertex] (x) at (0,0)  {};
			\node [vertex] (y) at (0.2,0)  {};
			
			\draw (x) to [bend right] (y);
			\draw (y) to [bend right] (x);%
		\end{tikzpicture}%
}}

\newcommand{\homsconeone}{{%
		\begin{tikzpicture}%
			\tikzstyle{vertex} =[circle,draw,fill=black,thick, inner sep=0pt,minimum size= .2 mm]
			
			\clip (-0.25,-0.1) rectangle (0.4, 0.1);
			
			\node [vertex] (x) at (0,0)  {};
			\draw (x) node [left=-0.8mm, font=\scriptsize] {$1$};
			\node [vertex] (y) at (0.2,0)  {};
			\draw (y) node [right=-0.8mm, font=\scriptsize] {$1$};
			
			\draw (x) to [bend right] (y);
			\draw (y) to [bend right] (x);%
		\end{tikzpicture}%
}}

\newcommand{\anyloop}{{%
		\begin{tikzpicture}%
			\tikzstyle{vertex} =[circle,draw,fill=black,thick, inner sep=0pt,minimum size= .3 mm]
			
			\clip (-0.22,-0.11) rectangle (0.03, 0.11);
			
			\node [vertex] (x) at (0,0)  {};
			
			\draw (x) arc (0:360:0.1);%
		\end{tikzpicture}%
}}

\newcommand{\mloop}{{%
		\begin{tikzpicture}%
			\tikzstyle{vertex} =[circle,draw,fill=black,thick, inner sep=0pt,minimum size= .3 mm]
			
			\clip (-0.22,-0.11) rectangle (0.3, 0.11);
			
			\node [vertex] (x) at (0,0)  {};
			\draw (x) node [right=-0.5mm, font=\scriptsize] {$m$};
			
			\draw (x) arc (0:360:0.1);%
		\end{tikzpicture}%
}}

\newcommand{\oneloop}{{%
		\begin{tikzpicture}%
			\tikzstyle{vertex} =[circle,draw,fill=black,thick, inner sep=0pt,minimum size= .3 mm]
			
			\clip (-0.22,-0.11) rectangle (0.19, 0.11);
			
			\node [vertex] (x) at (0,0)  {};
			\draw (x) node [right=-0.5mm, font=\scriptsize] {$1$};
			
			\draw (x) arc (0:360:0.1);%
		\end{tikzpicture}%
}}

\newcommand{\twoloop}{{%
		\begin{tikzpicture}%
			\tikzstyle{vertex} =[circle,draw,fill=black,thick, inner sep=0pt,minimum size= .3 mm]
			
			\clip (-0.22,-0.11) rectangle (0.2, 0.11);
			
			\node [vertex] (x) at (0,0)  {};
			\draw (x) node [right=-0.5mm, font=\scriptsize] {$2$};
			
			\draw (x) arc (0:360:0.1);%
		\end{tikzpicture}%
}}

\newcommand{\twotoone}{{%
		\begin{tikzpicture}%
			\tikzstyle{vertex} =[circle,draw,fill=black,thick, inner sep=0pt,minimum size= .3 mm]
			
			\clip (-0.22,-0.1) rectangle (0.35, 0.1);
			
			\node [vertex] (x) at (0,0)  {};
			\draw (x) node [left=-0.6mm, font=\scriptsize] {$2$};
			\node [vertex] (y) at (0.2,0) {};
			\draw (y) node [right=-0.8mm, font=\scriptsize] {$1$};

			\draw (x) to (y);
		\end{tikzpicture}%
}}

\newcommand{\twotostar}{{%
		\begin{tikzpicture}%
			\tikzstyle{vertex} =[circle,draw,fill=black,thick, inner sep=0pt,minimum size= .3 mm]
			
			\clip (-0.21,-0.1) rectangle (0.38, 0.1);
			
			\node [vertex] (x) at (0,0)  {};
			\draw (x) node [left=-0.6mm, font=\scriptsize] {$2$};
			\node [vertex] (y) at (0.2,0) {};
			\draw (y) node [right=-0.6mm, font=\scriptsize] {$\ast$};
			
			\draw (x) to (y);
		\end{tikzpicture}%
}}

\newcommand{\monetomtwo}{{%
		\begin{tikzpicture}%
			\tikzstyle{vertex} =[circle,draw,fill=black,thick, inner sep=0pt,minimum size= .3 mm]
			
			\clip (-0.5,-0.1) rectangle (0.6, 0.1);
			
			\node [vertex] (x) at (0,0)  {};
			\draw (x) node [left=-0.8mm, font=\scriptsize] {$m_1$};
			\node [vertex] (y) at (0.2,0) {};
			\draw (y) node [right=-0.8mm, font=\scriptsize] {$m_2$};
			
			\draw (x) to (y);
		\end{tikzpicture}%
}}

\newcommand{\closedgeod}{{%
		\begin{tikzpicture}%
			
			\tikzstyle{vertex} =[circle,draw,fill=black,thick, inner sep=0pt,minimum size= .3 mm]

			\clip (-0.12,-0.11) rectangle (0.13, 0.11);
			
			\node [vertex] (x) at (0.1,0)  {};
			
			\draw (-115:0.1) arc (-115:115:0.1);
			\draw[densely dotted] (115:0.1) arc (115:245:0.1);
			
		\end{tikzpicture}%
}}

  \usepackage{hyperref}
  \usepackage[noabbrev, capitalize]{cleveref}

  \newcommand{\calB}{\mathcal{B}}
  \newcommand{\calC}{\mathcal{C}}

  \newcommand{\calH}{\mathcal{H}}

  \newcommand{\calQ}{\mathcal{Q}}

  \newcommand{\EE}{\mathbb{E}}

  \newcommand{\NN}{\mathbb{N}}
  
  \newcommand{\PP}{\mathbb{P}}
  
  \newcommand{\RR}{\mathbb{R}}

  \newtheorem{theorem}{Theorem}[section]
  \newtheorem{proposition}[theorem]{Proposition}
  
  \newtheorem{lemma}[theorem]{Lemma}

  \newtheorem{question}[theorem]{Question}

  \theoremstyle{definition}
  \newtheorem{definition}[theorem]{Definition}

  \newtheorem*{claim*}{Claim}

  \newtheorem*{question*}{Question}
  \newtheorem*{answer*}{Answer}
  \newtheorem*{application*}{Application}

  \theoremstyle{remark}
  
  \newtheorem*{remark*}{Remark}

  \DeclareMathOperator{\Vol}{Vol}
  
  \DeclareMathOperator{\Cone}{Cone}

  \newcommand{\param}{{\mathchoice{\mkern1mu\mbox{\raise2.2pt\hbox{$
  \centerdot$}}
  \mkern1mu}{\mkern1mu\mbox{\raise2.2pt\hbox{$\centerdot$}}\mkern1mu}{
  \mkern1.5mu\centerdot\mkern1.5mu}{\mkern1.5mu\centerdot\mkern1.5mu}}}

  \renewcommand{\setminus}{{\smallsetminus}}

  \newcommand{\ST}{\mathbin{\Big|}} 
  \newcommand{\from}{\colon\thinspace} 
  \newcommand{\ep}{\epsilon}
   
  \newcommand{\One}{{\mathbbm{1}}}
  
  \newcommand{\Hgen}{{\calH_{\rm Gen}}}
  \newcommand{\Hng}{{\calH_{\rm NG}}}
  \newcommand{\sys}{{\rm sys}}

\begin{document}


  \title [Lengths of saddle connections]
   {Lengths of saddle connections \\ on random translation surfaces of large genus}
  

 \author   {Howard Masur}
 \address{Department of Mathematics, University of Chicago, Chicago, Il.}
 \email{masur@math.uchicago.edu}

 \author   {Kasra Rafi}
 \address{Department of Mathematics, University of Toronto, Toronto, ON }
 \email{rafi@math.toronto.edu}
 
 \author   {Anja Randecker}
 \address{Institute for Mathematics, Heidelberg University, Heidelberg}
 \email{randecker@mathi.uni-heidelberg.de}
 
 \date{September 8, 2025}

\begin{abstract} 
We determine the distribution of the number of saddle connections on a random translation surface of large genus.
More specifically, for genus $g$ tending to infinity, the number of saddle connections with lengths in a given interval $[\sfrac{a}{g}, \sfrac{b}{g}]$ converges in distribution to a Poisson distributed random variable. Furthermore, the numbers of saddle connections associated to disjoint intervals of lengths are independent.
\end{abstract}
  
  \maketitle

\section{Introduction}
We study the distribution of short saddle connections on a random large-genus translation surface. 
This is part of a continuing effort to understand the geometry of a random surface as the genus of 
the surface tends to infinity. Much is known about the shape of random hyperbolic surfaces
but most questions about random translation surfaces remain open. Here, we aim to emulate 
the celebrated results of Mirzakhani--Petri \cite{mirzakhani_petri_19} regarding the number of short 
curves in a random large-genus hyperbolic surface. 

The space of translation surfaces of genus $g$ can be identified with the space of abelian differentials 
$(X, \omega)$ where $X$ is a Riemann surface of genus $g$ and $\omega$ is a holomorphic $1$--form
on $X$. The space of abelian differentials can be decomposed into strata depending on the number and the 
order of zeros of $\omega$. Let $\calH_g=\calH_g(1,1, \dots, 1)$ be the principal stratum of unit-area abelian differentials on 
a closed surface of genus $g$ where there are~$2g-2$ zeros of order~$1$. 
 We equip the stratum with the normalized Lebesgue measure $\Vol(\param)$ as in~\cite{Masur-82, Veech-82}. 
 Then $\calH_g$ has a finite total volume \cite{Masur-82, Veech-82} and we can define the probability of a measurable subset $E \subseteq \calH_g$ by 
\[
\PP_g(E) = \frac{\Vol(E)}{\Vol(\calH_g)}. 
\]

Since the total area of a surface in $\calH_g$ is $1$ and the number of zeros of an abelian differential in $\calH_g$ is 
$2g-2$, the expected number of saddle connections of length at most $\ep$ on a random translation surface in 
$\calH_g$ tends to infinity as $g \to \infty$. Therefore, we need to choose the right scale at which the expected number 
of saddle connections is a finite, positive real number.  We show that this correct scale is 
$\sfrac 1g$ in the following sense.
Given $(X, \omega) \in \calH_g$ and an interval $[a, b] \subseteq \RR_+$, let $N_{g,[a,b]}(X,\omega)$ denote  the number of saddle connections 
on $(X, \omega)$ with lengths in the interval~$\left[ \sfrac ag, \sfrac bg \right]$.

\begin{theorem}  \label{Thm:Main}
Let $[a_1, b_1], [a_2, b_2], \dots, [a_k, b_k] \subseteq \RR_+$ be disjoint intervals. Then, as
$g \to \infty$, the vector of random variables 
\[
\left( N_{g,[a_1,b_1]},\dots ,N_{g,[a_k,b_k]} \right) \from \calH_g \to \NN_0^k
\]
converges jointly in distribution to a vector of random variables with Poisson distributions of means 
$\lambda_{[a_i,b_i]}$, where
\[
\lambda_{[a_i, b_i]}= 8\pi(b_i^2-a_i^2)
\]
 for $i = 1,\dots,k$. That is, 
\[
\lim_{g \to \infty} \PP_g \left( N_{g,[a_1,b_1]} = n_1,..., N_{g,[a_k,b_k]} = n_k \right) = 
\prod_{i=1}^k \frac{\lambda_{[a_i, b_i]}^{n_i} e^{-\lambda_{[a_i,b_i]}}}{n_i!}.
\]
 \end{theorem} 

\cref{Thm:Main} has the following consequence: Let $l_{\rm min}(X, \omega)$ be the length of the shortest saddle connection on $(X, \omega)$. Then
\begin{equation*} 
\lim_{g \to \infty} \PP_g \left( l_{\rm min} < \sfrac{\ep}{g} \right)
= 1 - \PP \left( N_{g,[0,\ep]} = 0 \right)
= 1 - e^{-\lambda_{[0,\ep]}}
\approx
 8\pi \ep^2
 \quad \text{as } \ep \to 0
 .
\end{equation*}

This is in line with the fact that for a fixed stratum $\calH$ of unit-area translation surfaces and a small $\ep > 0$, the thin part $\calH_{\rm thin}^{\ep} \subseteq \calH$ which contains all translation surfaces with a saddle connection of length at most~$\ep$ satisfies
\begin{equation*}
	\Vol(\calH_{\rm thin}^{\ep} ) = O(\ep^2) \cdot \Vol(\calH)
	.
\end{equation*}
For one short saddle connection, this can be proven using Siegel--Veech constants. However, also for the set of translation surfaces with two short saddle connections, there exists a similar result on asymptotics by \cite{masur_smillie_91} which can be made more concrete with \cref{Thm:Main}.

\subsection*{Other strata}
For a given genus, the principal stratum has the highest dimension and includes all other strata for that genus on 
its boundary. Thus, although \cref{Thm:Main} is stated for the principal stratum, it also applies to the space of all 
unit-area translation surfaces of a given genus. For non-principal strata, the result does not hold universally; 
however, in certain cases, similar conclusions can be drawn, which we now explore.

The first variation is for strata where the number of non-simple zeros is growing slow enough.

\begin{theorem} \label{thm:slow-growing_exceptions}
	Let $f, \ell \colon \NN \to \RR$ be functions such that $\frac{f(g)^2 \cdot \ell(g)}{g} \to 0$ for $g\to \infty$. If $\mathcal{H}_g$ is replaced by the stratum $\calH_g(m_1, m_2, \ldots, m_{\ell(g)}, 1, \ldots, 1)$ of unit-area translation surfaces of genus~$g$ with at most $\ell(g)$ zeros which are not simple and $m_1,\ldots, m_{\ell(g)} \leq f(g)$, then the same statement as in \cref{Thm:Main} is true for the same $\lambda_{[a_i, b_i]} = 8\pi (b_i^2 - a_i^2)$.
\end{theorem}

We can also consider strata where none of the zeros is simple but all zeros have the same, fixed~order. 

\begin{theorem} \label{thm:order_m}
	Fix $m\in \NN$ and consider only such $g\in \NN$ for which $2g-2$ is divisible by~$m$.	
	If $\mathcal{H}_g$ is replaced by the stratum $\calH_g(m, m, \dots, m)$ of unit-area translation surfaces of genus~$g$ where there are $\frac{2g-2}{m}$ zeros of order $m$, then the same statement as in \cref{Thm:Main} is true with means
	\[
	\lambda_{[a_i, b_i]}= \left( \frac{m+1}{m}\right)^2 \cdot 2\pi(b_i^2-a_i^2)
	\]
	for $i = 1,\dots,k$. 
\end{theorem}

Other slight variations of the theorem can also be derived with the same method.
However, our method does not work for the minimal stratum $\calH_g(2g-2)$ when there is only one zero. The reason is that for a surface in the minimal stratum, every saddle connection is  a closed loop and cannot be collapsed. However, the scale $\sfrac 1g$ 
still seems to be correct. 
\begin{question} 
	Is there  an analogue of \cref{Thm:Main} for the space $\calH_g(2g-2)$? 
\end{question}

\subsection*{History and tools}

Large-genus asymptotics in the context of surfaces have been studied for some years primarily for hyperbolic surfaces. Mirzakhani started a program in~\cite{mirzakhani_13} in which she determined the expected systole, Cheeger constant, diameter, and other geometric properties of random surfaces of large genus.
Independently, Guth--Parlier--Young~\cite{guth_parlier_young_11} considered geometric invariants related to pants decompositions of random hyperbolic surfaces.
Subsequently, many properties have been studied such as lengths of separating curves \cite{nie_wu_xue_23} or the first non-zero Laplacian eigenvalue \cite{wu_xue_22, lipnowski_wright_24, anantharaman_monk_24}.
The inspiration for the article at hand is the work of Mirzakhani--Petri \cite{mirzakhani_petri_19} in which they determined the distribution of the number of short curves in a random hyperbolic surface of large genus.

For translation surfaces, less is still known than for hyperbolic surfaces.
However, for the volumes of strata of translation surfaces, there has been a recent increase in results. Based on work of Eskin--Okounkov \cite{eskin_okounkov_01}, Eskin--Zorich predicted in \cite{eskin_zorich_15} the volumes of strata of translation surfaces and Aggarwal confirmed these \cite{aggarwal_20}. A more precise error term for the volume estimates has been determined by Chen--Möller--Sauvaget--Zagier~\cite{chen_moeller_sauvaget_zagier_20}. Also for half-translation surfaces, first results have been obtained by Chen--Möller--Sauvaget~\cite{chen_moeller_sauvaget_23}.

By work of Eskin--Masur--Zorich \cite{eskin_masur_zorich_03}, volumes of strata are strongly related to Siegel--Veech constants (see \cref{sec:Siegel-Veech} for definitions and more details). 
Asymptotics of Siegel--Veech constants for large genus have been calculated by many different authors in recent years, e.g.\  Chen--Möller--Zagier \cite{chen_moeller_zagier_18} for principal strata, Sauvaget \cite{sauvaget_18} for minimal strata, Zorich in an appendix to \cite{aggarwal_20} and Aggarwal \cite{aggarwal_19} for all strata of translation surfaces, and Chen--Möller--Sauvaget--Zagier \cite{chen_moeller_sauvaget_zagier_20} for principal strata of half-translation surfaces.

Furthermore, also geometric properties of random translation surfaces of large genus have been studied: An upper bound for the covering radius is given in \cite{masur_rafi_randecker_22} and Delecroix--Goujard--Zograf--Zorich have determined geometric and combinatorial properties for large-genus square-tiled surfaces in a series of articles, see \cite{delecroix_goujard_zograf_zorich_22, delecroix_goujard_zograf_zorich_23} and the references therein. In an upcoming work, Bowen--Rafi--Vallejos show that a sequence of random translation surfaces with area equal to genus Benjamini--Schramm converges as genus tends to infinity, see the research announcement \cite{bowen_rafi_vallejos_25}.

For the proof of our main theorem, we will strongly use the recent works on large-genus asymptotics for volumes of strata and for Siegel--Veech constants that are outlined above.

\subsection*{Further questions and remarks}
Let $\sys(X, \omega)$ be the length of the shortest closed geodesic in~$(X, \omega)$.
To examine $\sys$, we need to work at a different scale as the expected number of closed geodesics of length $\sfrac 1g$ converges to zero. 
It turns out that the correct scale for this setup is~$\sfrac 1{\sqrt g}$. Let $L_{g,[a,b]}(X,\omega)$ denote  the number of closed geodesics on~$(X, \omega)$ with lengths in the interval $\left[ \sfrac a{\sqrt g}, \sfrac b{\sqrt g} \right]$.
\begin{question} \label{question:closed_geodesics}
	Does $L_{g,[a,b]}(X,\omega)$ converge in distribution to a random variable with a Poisson distribution?
\end{question}  

We can not directly apply the same methods as in this paper, again because our methods involve the collapsing of saddle connections which does not work for closed geodesics. 

Another natural setting is the space of quadratic differentials (or half-translation surfaces). 
Let $\calQ_g= \calQ_g(1,1, \dots, 1)$ be the principal stratum of unit-area quadratic differentials. 
Again the correct scale seems to be $\sfrac 1g$. 

\begin{question} 
Does an analogue of \cref{Thm:Main} and/or an analogue of \cref{question:closed_geodesics} hold for the space $\calQ_g$? 
\end{question} 
Again, our methods fail for the analogue of \cref{Thm:Main} since collapsing a saddle connection in a quadratic differential is not a local construction. Furthermore, much less  is
known about the volumes of strata of quadratic differentials. 

\subsection*{Outline of the paper}
In the next four sections, we recall the background needed for our arguments:
In \cref{sec:moments}, we describe the method of moments which is used to deduce that a variable is Poisson distributed. In \cref{sec:translation_surfaces_strata}, we introduce translation surfaces and their strata. The method of Siegel--Veech constants and some relevant values of them are collected in \cref{sec:Siegel-Veech}. In \cref{sec:collapsing}, we recall the construction for collapsing a saddle connection and describe a variation of it to collapse two saddle connections sharing a zero.

We start the proof of \cref{Thm:Main} in \cref{sec:generic} by defining a generic subset of the stratum, on which we describe a simultaneous collapsing procedure in \cref{sec:collapsing_simultaneously}. This construction is used in \cref{sec:length_spectrum} to determine the factorial moments of $N_{g,[a,b]}$ on the generic subset. In \cref{sec:finish}, we explain how to extend the result from the generic subset to the whole stratum.

We finish with arguments to prove the variations of the main theorem in \cref{sec:variations}.

\subsection*{Acknowledgements}
The authors are very grateful to Anton Zorich for many helpful conversations during the work on this project as well as for factors of $2$ and detailed comments on an earlier draft of this work.

K.~R.\ was partially supported by NSERC Discovery Grant RGPIN 05507 and the Simons Fellowship Program. The work of A.~R.\ was partially funded by the Deutsche Forschungsgemeinschaft
(DFG, German Research Foundation) -- 441856315.

\section{The method of moments} \label{sec:moments}
 Similar to the proof in \cite{mirzakhani_petri_19}, we use the method of moments for showing that a random variable is Poisson distributed.
 Recall that a random variable $N \from \Omega \to \NN_0$ on a probability space $(\Omega, \PP)$ is said to be \emph{Poisson distributed with mean~$\lambda \in (0, \infty)$} if
 \[
 \PP ( N=k ) = \frac{\lambda^k e^{-\lambda}}{k!} \qquad \text{for all $k \in \NN_0$}. 
 \]
 The moments of a Poisson distributed random variable are slightly complicated terms, so we instead consider the factorial moments here.
 
Given a random variable $N \from \Omega \to \NN_0$ and $r \in \NN$, we define the random variable
\[
(N)_r =N(N-1)\dots(N-r+1).
\]
If its expectation $\EE (N)_r$ exists, it is called the \emph{$r$--th factorial moment} of $N$.

The $r$--th factorial moment of a Poisson distributed variable with mean $\lambda$ is equal to $\lambda^r$ for all $r \in \NN$. 
It turns out that this is an if and only if condition. 

\begin{theorem}[The method of moments {\cite[Theorem 1.23]{bollobas_01}}] \label{Thm:moments}
Let $\{(\Omega_i,\PP_i)\}_{i\in\NN}$ be a sequence of probability spaces. For $k \in \NN$, 
let $N_{1,i}, \dots , N_{k,i} \from  \Omega_i \to \NN_0$ be random variables for all $i\in \mathbb{N}$ and suppose 
there exist $\lambda_1, \dots , \lambda_k \in (0, \infty)$ such that
\[
\lim_{i\to \infty} \EE ( (N_{1,i})_{r_1} \dots (N_{k,i})_{r_k} ) 
  = \lambda_1^{r_1}\dots \lambda_k^{r_k}
\]
for all $r_1,\dots, r_k \in \NN$.  Then
\[
\lim_{i \to \infty} \PP ( N_{1,i} = n_1,...,N_{k,i} = n_k ) = 
\prod_{j=1}^k \frac{\lambda_j^{n_j} e^{-\lambda_j}}{n_j!}
\]
for all $n_1,\dots , n_k \in \NN$. In other words, the vector 
$(N_{1,i},\dots,N_{k,i}) \from \Omega_i \to \NN_0^k$ converges jointly in distribution to a vector 
which is independently Poisson distributed with means~$\lambda_1,\dots,\lambda_k$.
\end{theorem}

\section{Translation surfaces and strata} \label{sec:translation_surfaces_strata}

A \emph{translation surface} $(X, \omega)$ is a pair consisting of a compact, connected Riemann surface~$X$ and a holomorphic $1$--form $\omega$ on $X$. By integrating the $1$--form, we obtain a metric on $X$ which is locally isometric to the Euclidean plane except in a neighbourhood of the zeros of $\omega$. The neighbourhood of a zero of order $k$ is isometric to a neighbourhood of the branching point of a $(k+1)$--covering of a Euclidean disk.

Instead of considering the moduli space of all translation surfaces of a given genus~$g$, we consider finer spaces with more information: For a \emph{stratum}, we focus only on those translation surfaces whose area is $1$ and where the zeros are labelled and their orders are fixed. Given a partition of $2g-2 = \sum_{i=1}^\ell m_i$, the \emph{stratum} $\calH_g( m_1, \ldots, m_\ell)$ is the space of all unit-area translation surfaces of genus $g$ with $\ell$ zeros of order $m_1, \ldots, m_\ell$.

A \emph{saddle connection} in $(X,\omega)$ is a geodesic segment that starts and ends in a zero and does not contain a zero in its interior. 
To any given (oriented) saddle connection $\gamma$ on~$(X, \omega)$, we associate the vector~$\int_\gamma \omega$ in $\mathbb{C}$, called the \emph{holonomy vector} of the saddle connection.
On a given translation surface, the number of saddle connections grows quadratically in their length \cite{masur_88, masur_90}.

A saddle connection can also be thought of as an element of the relative homology group of~$(X, \omega)$ relative to the set of zeros. If $\calB$ is a set of saddle connections that form an integral basis for the relative homology, then the holonomy vectors of the elements of $\calB$ determine~$(X, \omega)$. For every $(X,\omega)$ and $\calB$, there is a neighbourhood $U$ of  $(X, \omega)$ in the stratum (here, stratum is meant without the requirement that the area is $1$) such that for every $(X', \omega')$ in~$U$, all elements of~$\calB$ (thought of as elements in the relative homology group) can still be represented in $(X', \omega')$ as saddle connections. Then the set of holonomy vectors of saddle connections in $\calB$ give coordinates for translation surfaces in $U$.
We refer to this set of holonomy vectors as \emph{period coordinates} around $(X,\omega)$.

The period coordinates give an embedding from $U$ to $\mathbb{C}^{2g+\ell-1}$. The associated pullback measure, denoted $\mu$,  of Lebesgue measure on $\mathbb{C}^{2g+\ell-1}$, on the moduli space was studied by Masur \cite{Masur-82} and Veech \cite{Veech-82}. It also defines a measure on the stratum $\calH_g(m_1,\ldots,m_\ell)$ in the following way. For an open set $V \subseteq \calH_g(m_1,\ldots,m_\ell)$, its volume is defined by
\begin{equation}
	\label{eq:def_MV-measure}
	\Vol(V) \coloneqq
	2(2g + \ell -1)
	\cdot \mu(\Cone(V)) 
\end{equation}
where $\Cone(V)$ is the cone of $V$ in the moduli space and $2(2g + \ell -1)$ is the real dimension of the cone. We refer to this measure on $\calH_g(m_1,\ldots,m_\ell)$ when we write $\Vol(\param)$ in the following discussion.

Calculating volumes of strata explicitly is difficult in general. However, using an algorithm of Eskin--Okounkov \cite{eskin_okounkov_01}, Eskin--Zorich predicted in \cite{eskin_zorich_15} and Aggarwal confirmed in \cite[Theorem 1.4]{aggarwal_20} that we have
\begin{equation} \label{eq:volume_stratum}
	\Vol(\calH_g( m_1, \ldots, m_\ell) ) = \frac{4}{\prod_{i=1}^\ell (m_i + 1)} \cdot \left( 1+ O\left(\frac1g \right) \right)
	.
\end{equation}

Recall that strata are not necessarily connected, they can have up to three connected components \cite{kontsevich_zorich_03}.

\section{The method of Siegel--Veech constants} \label{sec:Siegel-Veech}

A useful method to calculate the probability of having a saddle connection of a certain length is by Siegel--Veech constants. To do this, we fix a stratum $\calH$ of unit-area translation surfaces. For $X \in \calH$, $V_{sc}(X)$ denotes the set of (holonomy vectors of) saddle connections on $X$. A~configuration $\calC$ of a saddle connection includes the multiplicity of the saddle connection, the order of the zeros that the saddle connection connects, and whether the zeros are different or not.
Let $V_{\calC}(X) \subseteq V_{sc}(X)$ denote the subset of holonomy vectors whose saddle connections are in the configuration $\calC$. For an integrable function $f:\RR^2 \rightarrow \RR$ with compact support, we define the \textit{Siegel--Veech transform} $\hat f_{\calC}:\calH \rightarrow \RR$ via
     \[
       \hat f_{\calC} (X) = \sum_{v \in V_{\calC}(X)} f(v).
     \]

\begin{theorem}[Siegel--Veech formula {\cite{veech_98}}]
	\label{thm:siegel-veech-formula}
	Let $\calH$ be a connected component of a stratum of unit-area translation surfaces and~$\calC$ a configuration of saddle connections. Then there exists a constant~$c(\calC, \calH)$ such that for every integrable $f:\RR^2 \rightarrow \RR$ with compact~support,
	\begin{equation*}
		\mathbb{E} \left( \hat f_{\calC} \right) =  \frac{\int_{\calH} \hat f_{\calC}}{\Vol(\calH)} = c(\calC, \calH) \cdot \int_{\RR^2}f
		.
	\end{equation*}
 	The integral on the left is with respect to the measure on $\calH$ defined in \protect\eqref{eq:def_MV-measure}, and the integral on the right is with respect to Lebesgue measure on $\RR^2$.
\end{theorem}

The constant $c(\calC, \calH)$ is called the \emph{Siegel--Veech constant} of the configuration and the stratum. It does not only appear when studying the average number of saddle connections for translation surfaces in a given stratum but also when counting saddle connections on a fixed translation surface of this stratum.

\begin{theorem}[Quadratic growth of number of saddle connections \cite{eskin_masur_01}]
	Let $\calH$, $\calC$, and~$c(\calC, \calH)$ as in \cref{thm:siegel-veech-formula}. Then for almost every $(X,\omega) \in \calH$, we have
	\begin{equation*}
		\lim_{T \to \infty} \frac{ \left| V_{\calC}(X) \cap B(0,T) \right| }{\pi T^2}
		= c(\calC, \calH)
		.
	\end{equation*}
\end{theorem}

\Cref{thm:siegel-veech-formula} allows us to estimate expected values of various quantities on a stratum in terms of Siegel--Veech constants.  Eskin--Masur--Zorich developed in \cite{eskin_masur_zorich_03} a method for computing Siegel--Veech constants using combinatorial constructions and surgeries as well as the Siegel--Veech formula above.
Based on their recursive formulas, asymptotics of Siegel--Veech constants for large genus have been more recently calculated by many different authors, e.g.\  \cite{chen_moeller_zagier_18, sauvaget_18, aggarwal_20, aggarwal_19, chen_moeller_sauvaget_zagier_20}. We recall now the Siegel--Veech constants for the configurations that are relevant for our proofs.

Let $\calC_{\! \monetomtwo}$ be the configuration of a single saddle connection connecting a fixed zero of order $m_1$ to a fixed, different zero of order $m_2$. Then by \cite[Corollary 1]{aggarwal_20} and \cite[Theorem 1.2]{aggarwal_19} for connected strata and \cite[Equation (14)]{randecker_25} for non-connected strata, the corresponding Siegel--Veech constant is
\begin{equation} \label{Eq:sc}
	c_{\! \monetomtwo} = (m_1+1)(m_2+1) \cdot \left(1 + O \left( \frac{1}{g} \right) \right).
\end{equation}

Let $\calC_{\! \homsconeone}$ be the configuration of two homologous saddle connections connecting a fixed zero of order $1$ to a fixed, different zero of order $1$. Then by \cite[Proposition 3.4]{aggarwal_19} and \cite[Theorem 10]{randecker_25}, the corresponding Siegel--Veech constant is
\begin{equation} \label{Eq:hom}
	c_{\! \homsconeone} = O \left(\frac{1}{g^2} \right).
\end{equation}

Let $\calC_\mloop$ be the configuration of a saddle connection connecting a fixed zero of order~$m$ to itself. Then by \cite[Appendix 1]{vallejos_24} and \cite[Corollary 7.4]{randecker_25}, the corresponding Siegel--Veech constant is 
\begin{equation} \label{eq:loop}
	c_{\mloop} = (m+1)^2 \cdot \left(1 + O \left( \frac{1}{g} \right) \right)
	.
\end{equation}

 \section{Collapsing and opening up zeros} \label{sec:collapsing}
Our calculations depend heavily on the estimation of volumes of subsets of strata which contain only translation surfaces with short saddle connections.
For this, we vary a construction from \cite[Sections 8.1 and 8.2]{eskin_masur_zorich_03} that collapses a saddle connection between two different zeros of order $1$. We describe here the construction and its inverse.

Let $(X,\omega)$ be a translation surface with two zeros $v_1$ and $v_2$ of order $1$ that are connected by a saddle connection $\alpha$ of length $\delta$. We assume for the description of the construction that~$\alpha$ is horizontal.
There are four horizontal separatrices starting in $v_2$, one of them being~$\alpha$.
Let $\gamma$ be a geodesic segment of length $\delta$ on the horizontal separatrix starting in~$v_2$ which forms an angle of $2\pi$ with $\alpha$. We call $\gamma$ the \emph{ghost double} of $\alpha$.
Note that $\gamma$ might not always exist and we will discuss in detail later what to do when it does not exist.

We now cut open along $\alpha$ and $\gamma$, obtaining two copies of $\alpha$ and $\gamma$ each in the metric completion of $X \setminus (\alpha \cup \gamma)$. The upper copy of $\alpha$ we call $\alpha_1$ and the lower copy~$\alpha_2$, correspondingly for $\gamma$. Then we reglue $\alpha_1$ with~$\gamma_1$ and reglue $\alpha_2$ with $\gamma_2$ as in \cref{fig:collapsing_one}.

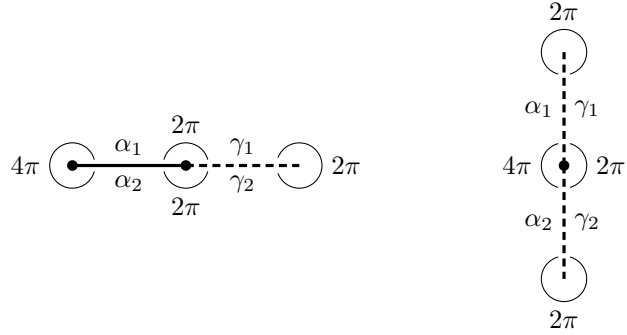
\begin{figure}[h]
	\centering
	\begin{tikzpicture}[scale=1]
		\draw[very thick] (-1.5,0) -- node[above]{$\alpha_1$} node[below]{$\alpha_2$} (0,0) ;
		\draw[very thick, densely dashed] (0,0) -- node[above]{$\gamma_1$} node[below]{$\gamma_2$} (1.5,0);
		
		\fill (-1.5,0) circle (2pt);
		\draw (-1.5,0) + (-15:0.3) arc (-15:-180:0.3) node[left]{$4\pi$} arc (180:15:0.3);
		
		\fill (0,0) circle (2pt);
		\draw (165:0.3) arc (165:90:0.3) node[above] {$2\pi$}  arc (90:15:0.3);
		\draw (-165:0.3) arc (-165:-90:0.3) node[below] {$2\pi$}  arc (-90:-15:0.3);
		
		\draw (1.5,0) + (160:0.3) arc (160:0:0.3) node[right]{$2\pi$} arc (0:-160:0.3);
		
		\begin{scope}[xshift=5cm]
			\draw[very thick, densely dashed] (0,0) -- node[left]{$\alpha_1$} node[right]{$\gamma_1$} (0,1.5) ;
			\draw[very thick, densely dashed] (0,0) -- node[left]{$\alpha_2$} node[right]{$\gamma_2$} (0,-1.5);
			
			\fill (0,0) circle (2pt);
			\draw (105:0.3) arc (105:180:0.3) node[left]{$4\pi$} arc (-180:-105:0.3);
			\draw (-75:0.3) arc (-75:0:0.3) node[right]{$2\pi$} arc (0:75:0.3);
			
			\draw (0,1.5) + (-75:0.3) arc (-75:90:0.3) node[above]{$2\pi$} arc (90:255:0.3);
			
			\draw (0,-1.5) + (75:0.3) arc (75:-90:0.3) node[below]{$2\pi$} arc (-90:-255:0.3);
		\end{scope}	
	\end{tikzpicture}
	\caption{Collapsing a zero into another zero.}
	\label{fig:collapsing_one}
\end{figure}

Note that the two points corresponding to $v_2$ under this cutting-and-gluing are regular points and that~$v_1$ turns into a new zero of order $2$. Therefore we call this construction \emph{collapsing the zero~$v_2$ into the zero $v_1$ (along $\alpha$)} and the inverse of it \emph{opening up a zero of order $2$}. The data for opening up a zero of order $2$ is the holonomy vector of the saddle connection $\alpha$ and a combinatorial information arising from choosing two out of the three geodesic segments with the given holonomy vector starting in the given zero of order $2$.

We consider now $\calH_g=\calH_g(1,1, \dots, 1)$, the principal stratum of unit-area translation surfaces of genus $g$ with~$2g-2$ zeros of order~$1$.
The process of collapsing a zero into another zero defines a map
\begin{equation} \label{eq:measure-preserving-map}
	\calH^\epsilon \times \{\pm 1\} \to \calH_g(2,1, \ldots, 1) \times M \times \mathbb{D}_{\epsilon}
\end{equation}
where $\calH^\epsilon$ is the subset of $\calH_g$ of translation surfaces with exactly one saddle connection of length at most~$\epsilon$ and the choice of $1$ or $-1$ indicates the orientation of this saddle connection.
Furthermore, $M$ is the set of possible combinatorial data of size $(2g-2)(2g-3)$ which records the names of the zeros that were connected by the saddle connection, and~$\mathbb{D}_{\epsilon}$ is the disk of radius~$\epsilon$ in $\mathbb{R}^2$.
This map is well-defined and its inverse, defined by opening up the unique zero of order $2$, is also well-defined up to the three possible choices of two geodesic segments. Hence, the map is $3$--to--$1$.
Furthermore, the map is locally measure-preserving: Let $(X,\omega) \in \calH^\epsilon$ and $\calB$ be an integral basis of relative homology for $(X, \omega)$ which contains the unique shortest saddle connection~$\alpha$. Then $\calB$ defines period coordinates around $(X,\omega)$ and the image of $\calB \setminus \alpha$ defines period coordinates around the translation surface in the image of~$(X,\omega)$ whereas the holonomy vector of $\alpha$ defines coordinates for $\mathbb{D}_{\epsilon}$. As the real dimensions of domain and image are both $2(4g - 3)$ and the lengths of the vectors in a basis of domain and image coincide, also the measures of a neighbourhood of $(X,\omega)$ and its image coincide; see \cite[Section~8.2]{eskin_masur_zorich_03} for more details.

If we consider several saddle connections which do not intersect (including not sharing a zero as end point) and whose ghost doubles do not intersect each other nor the saddle connections, we can collapse all pairs of zeros simultaneously, see \cref{sec:collapsing_simultaneously}.
For intersecting saddle connections, this is in general not true. However, for a specific configuration, we can use the following simultaneous collapsing~procedure.

Let $(X,\omega)$ be a translation surface with~$v_1$, $v_2$, and $v_3$ three zeros of order $1$, $\alpha$ a saddle connection between $v_1$ and~$v_2$ of length~$\delta_1$, and $\beta$ a saddle connection between $v_2$ and $v_3$ of length $\delta_2$.
To collapse $v_1$ and $v_3$ into $v_2$, we consider a ghost double of~$\alpha$, starting in~$v_1$ and forming angle $2\pi$ with $\alpha$, and a ghost double of~$\beta$, starting in $v_3$ and forming angle~$2\pi$ with~$\beta$. Note that we have to make sure that neither of the ghost doubles and saddle connections intersect. The cases, in which they do intersect, we again exclude later by other volume~estimations.

We cut open along $\alpha$, $\beta$, and their ghost doubles and reglue as indicated in \cref{fig:collapsing_two}. Then~$v_1$ and $v_3$ turn into regular points and $v_2$ turns into a zero of order $3$.

\begin{figure}[h]
	\centering
	\begin{tikzpicture}[scale=1]
		\draw[very thick] (0,0) -- node[above]{$\alpha_1$} node[below]{$\alpha_2$} (1.7,0) ;
		\draw[very thick, densely dashed] (1.7,0) -- node[above]{$\gamma_1$} node[below]{$\gamma_2$} (3.4,0);
		
		\draw[very thick] (0,0) -- node[left]{$\beta_1$} node[right]{$\beta_2$} (120:2.3) ;
		\draw[very thick, densely dashed] (120:2.3) -- node[left]{$\gamma_3$} node[right]{$\gamma_4$} +(120:2.3);
		
		\fill (0,0) circle (2pt);
		\draw (-15:0.3) arc (-15:-120:0.3) node[left]{$4\pi-\theta$} arc (-120:-225:0.3);
		\draw (15:0.3) arc (15:60:0.3) node[above right]{$\theta$} arc (60:105:0.3);
		
		\fill (1.7,0) circle (2pt);
		\draw (1.7,0) + (165:0.3) arc (165:90:0.3) node[above] {$2\pi$}  arc (90:15:0.3);
		\draw (1.7,0) + (-165:0.3) arc (-165:-90:0.3) node[below] {$2\pi$}  arc (-90:-15:0.3);
		
		\draw (3.4,0) + (160:0.3) arc (160:0:0.3) node[right]{$2\pi$} arc (0:-160:0.3);
		
		\fill (120:2.3) circle (2pt);
		\draw (120:2.3) + (-45:0.3) arc (-45:30:0.3) node[right]{$2\pi$} arc (30:105:0.3);
		\draw (120:2.3) + (-75:0.3) arc (-75:-150:0.3) node[left]{$2\pi$} arc (-150:-225:0.3);
		
		\draw (120:4.6) + (-45:0.3) arc (-45:120:0.3) node[above]{$2\pi$} arc (-240:-75:0.3);
		
		\begin{scope}[xshift=7cm, yshift=2cm]
			\draw[very thick, densely dashed] (0,0) -- node[pos=0.6, above]{$\alpha_1$} node[pos=0.6, below]{$\gamma_1$} +(0:1.7) ;
			\draw[very thick, densely dashed] (0,0) -- node[left]{$\alpha_2$} node[right]{$\gamma_2$} (0,-1.7);
			
			\draw[very thick, densely dashed] (0,0) -- node[above]{$\gamma_3$} node[below]{$\beta_1$} +(150:2.3) ;
			\draw[very thick, densely dashed] (0,0) -- node[left]{$\gamma_4$} node[right]{$\beta_2$} +(60:2.3);
			
			\fill (0,0) circle (2pt);
			\draw (-195:0.3) arc (-195:-150:0.3) node[left]{$4\pi-\theta$} arc (-150:-105:0.3);
			\draw (-75:0.3) node[right]{$2\pi$} arc (-75:-45:0.3) arc (-45:-15:0.3);
			\draw (135:0.3) arc (135:105:0.3) node[above]{$2\pi$} arc (105:75:0.3);
			\draw (45:0.3) node[right]{$\theta$} arc (45:30:0.3) arc (30:15:0.3);
			
			\draw (0:1.7) + (-165:0.3) arc (-165:0:0.3) node[right]{$2\pi$} arc (0:165:0.3);
			
			\draw (0,-1.7) + (75:0.3) arc (75:-90:0.3) node[below]{$2\pi$} arc (-90:-255:0.3);
			
			\draw (150:2.3) + (-15:0.3) arc (-15:150:0.3) node[left] {$2\pi$} arc (150:315:0.3);
			
			\draw (60:2.3) + (-105:0.3) arc (-105:60:0.3) node[above] {$2\pi$} arc (60:235:0.3);
		\end{scope}	
	\end{tikzpicture}
	\caption{Collapsing two zeros into a third one, along two saddle connections.}
	\label{fig:collapsing_two}
\end{figure}
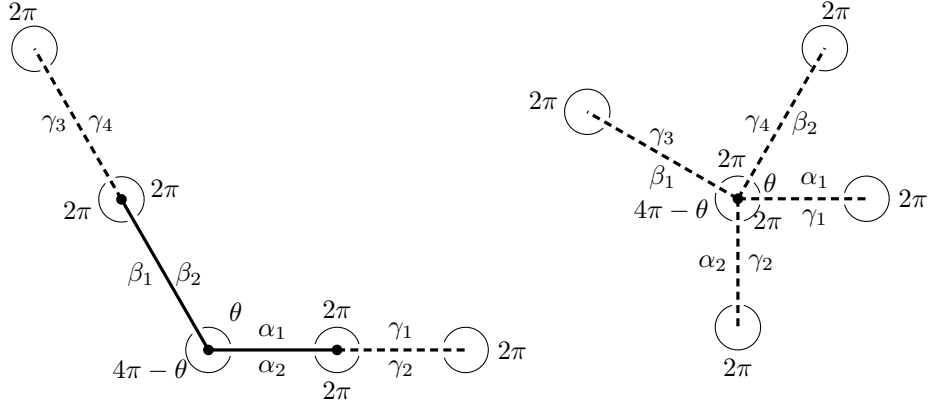

\section{The generic and the non-generic subsets of \texorpdfstring{$\calH_g$}{the stratum}} \label{sec:generic}

Our strategy to prove \cref{Thm:Main} is to show the statement first for a generic subset $\Hgen \subseteq \calH_g$. Then we show that the proportion of the measure of the non-generic subset~$\Hng$ in $\calH_g$ tends to zero as~$g \to \infty$. 

\begin{definition} \label{Def:generic}
Fix $B \in \mathbb{R}_+$.
We define the following subsets of $\calH_g$
\begin{align*} 
	\calH_\anyloop & \coloneqq
	\left\{
		(X,\omega) \in \calH_g : \,
		 X  \text{ contains a loop of length at most  $\sfrac{6B}{g}$}
	\right\} \\
	\calH_\cherry &\coloneqq 
	\left\{ \begin{aligned}
		(X,\omega) \in \calH_g : \, & X \text{ contains a saddle connection of length at most $\sfrac{B}{g}$ and} \\
		& \qquad \text{a saddle connection of length at most $\sfrac{2B}{g}$ that share a zero}
	\end{aligned}
	\right\}
\end{align*}
where a \emph{loop} is a saddle connection that connects a zero to itself.

Furthermore, we define the \emph{non-generic subset} of $\calH_g$ to be $\Hng = \calH_{\anyloop} \cup \calH_{\cherry}$ and the 
\emph{generic subset} to be the complement of the non-generic subset: $\Hgen = \calH_g \setminus \Hng$. 
\end{definition}

The next proposition justifies calling the subsets generic and non-generic.

\begin{proposition} \label{Prop:generic-is-generic}
Let  $\Hgen$ as in \cref{Def:generic}. Then
\begin{equation*}
	\frac{\Vol(\Hgen)}{\Vol(\calH_g)} \to 1 \qquad\text{as}\qquad g \to \infty. 
\end{equation*}
More specifically, 
\[
 \frac{\Vol(\calH_{\rm NG})}{\Vol(\calH_g)} = O\left( \frac1g \right)
\]
where the implied constant does not depend on $g$.
\end{proposition} 

\begin{proof}
	We first use the method of Siegel--Veech constants to obtain an estimate on the volume of $\calH_{\anyloop}$. For this, let $f \colon \RR^2 \to \RR$ be the indicator function of the ball of radius $\sfrac{6B}{g}$ centred at~$0$. Choose~$\calC_\oneloop$ to be the configuration of a saddle connections which forms a loop, starting at a zero of order $1$.
	
	Applying \cref{thm:siegel-veech-formula} and the Siegel--Veech constant from Equation~\eqref{eq:loop}, we obtain
	\begin{equation}
		\Vol (\calH_\anyloop) 
		= \int_{\calH_\anyloop } \One
		\leq \int_{\calH_g} \hat{f}_{\calC_\oneloop}
		= (2g-2) \cdot
		c_\oneloop \cdot \int_{\RR^2} f \cdot \Vol(\calH_g)
		= O\left( \frac{B^2}{g} \right)  \cdot \Vol(\calH_g)
		, \label{eq:volume_loop}
	\end{equation}
	where the factor $2g-2$ corresponds to the number of possible choices for the zero at which the loop starts and ends.
	
	We will later in the proof need similar estimates as \eqref{eq:volume_loop} for the following subset of $\calH_g$
	\begin{align*} 
		\calH_\homsc &\coloneqq
		\left\{ \begin{aligned}
			(X,\omega) \in \calH_g : \, &X \text{ contains a pair of homologous saddle connections }\\
			& \qquad \text{of length at most $\sfrac{B}{g}$ between different zeros}
		\end{aligned}\right\}
	\end{align*}
	and the following two subsets of $\calH_g(2,1,\ldots,1)$
	\begin{align*} 
		\calH_{\twoloop} & \coloneqq
		\left\{ \begin{aligned}
			(X,\omega) \in \calH_g(2,1,\ldots,1) : \, & X  \text{ contains a loop of length at most  $\sfrac{6B}{g}$} \\
			& \qquad \text{starting at the zero of order } 2
		\end{aligned} \right\} \\
		\calH_\twotoone &\coloneqq 
		\left\{ \begin{aligned}
			(X,\omega) \in \calH_g(2,1,\ldots,1) : \, & X  \text{ contains a saddle connection of length at most  $\sfrac{3B}{g}$} \\
			& \qquad \text{from the zero of order } 2 \text{ to a zero of order } 1
		\end{aligned} \right\}
		.
	\end{align*}
	
	Analogously to \eqref{eq:volume_loop}, we obtain with the Siegel--Veech constant from ~\eqref{Eq:hom}
	\begin{align}
		\Vol (\calH_\homsc) & = O\left( \frac{B^2}{g^2} \right)  \cdot \Vol(\calH_g)
		, \label{eq:volume_hom} \\
		\intertext{	with the Siegel--Veech constant from~\eqref{eq:loop} for $m=2$}
		\Vol ( \calH_{\twoloop} ) & = O \left( \frac{B^2}{g^2} \right)  \cdot \Vol(\calH_g(2,1,\ldots,1)) , \label{eq:volume_loops_order-2} \\
		\intertext{and with the Siegel--Veech constant from~\eqref{Eq:sc} for $m_1=2$ and $m_2=1$}
		\Vol (\calH_\twotoone ) &=  O \left( \frac{B^2}{g} \right)  \cdot \Vol(\calH_g(2,1,\ldots,1)) . \label{eq:volume_order-2}
	\end{align}

	To estimate the volume of $\calH_\cherry$, the idea is to simultaneously collapse the two saddle connections that share a zero. We introduce another subset of $\calH_g$  which contains some of the cases for which the collapsing is not possible. Let
	\begin{equation*}
		\calH_{\closedgeod} \coloneqq
		\left\{
			(X,\omega) \in \calH_g : 
			X \text{ contains a closed geodesic of length at most } \sfrac{6B}{g}
		\right\}
		.
	\end{equation*}
	
	Note that a closed geodesic on a translation surface in $\calH_\closedgeod$ is either a chain of saddle connections with angles at least $\pi$ between them or is contained in a cylinder. In the latter case, we can represent the closed geodesic as a boundary curve of the cylinder and hence we have a chain of saddle connections again.

	We already have an upper bound for the volume of $\calH_\anyloop$. To determine an upper bound for the volume of $\Hng = \calH_\anyloop \cup \calH_{\cherry}$, we consider now $\calH_{\cherry} \setminus (\calH_\anyloop \cup \calH_\closedgeod)$, $\calH_{\closedgeod} \setminus (\calH_\anyloop \cup \calH_\homsc)$ and~$\calH_\homsc$ separately.
	
	First, for a translation surface in $\calH_\cherry \setminus (\calH_\anyloop \cup \calH_\closedgeod)$, let $\alpha$ be the shortest saddle connection that shares a zero with a saddle connection of length at most $\sfrac{2B}{g}$. Furthermore, let $\beta$ be the shortest saddle connection that $\alpha$ shares a zero with.  Then we can collapse two zeros along $\alpha$ and $\beta$ into the shared zero. 
	As the length of $\alpha$ is at most~$\sfrac{B}{g}$ and the length of $\beta$ is at most $\sfrac{2B}{g}$,
	this gives us the following locally measure-preserving~map:
	\begin{equation*}
		\calH_{\cherry} \setminus (\calH_\anyloop \cup \calH_\closedgeod) \to M_{\cherry} \times \calH_g(3,1,\ldots, 1) \times \mathbb{D}_{\sfrac{B}{g}} \times \mathbb{D}_{\sfrac{2B}{g}} 
	\end{equation*}
	Note that this map is well-defined on a subset of full measure:
	The saddle connections $\alpha$ and~$\beta$ are uniquely determined except for a subset of measure $0$.
	And if $\alpha$ and $\beta$	or their ghost doubles would intersect, $\alpha$ would be a loop or we would have a closed geodesic of length at most~$\sfrac{6B}{g}$ that shares a zero with $\alpha$, so the translation surface would be in $\calH_\anyloop \cup \calH_\closedgeod$.
	The set~$M_{\cherry}$ here expresses the possible choices of the three zeros involved and hence is of order $g^3$.
	We can define an inverse of the map up to a finite choice of the geodesic segments which are to be cut open. Hence, the map is finite--to--$1$ and we obtain with Equation~\eqref{eq:volume_stratum}
	\begin{align*}
		\frac{ \Vol (\calH_{\cherry} \setminus  (\calH_\anyloop \cup \calH_\closedgeod) )}{ \Vol (\calH_g) }
		& = O(1) \cdot |M_{\cherry}|
		\cdot \frac{ \Vol (\calH_g(3,1,\ldots, 1))}{ \Vol (\calH_g) }
		\cdot \Vol \left( \mathbb{D}_{\sfrac{B}{g}} \times \mathbb{D}_{\sfrac{2B}{g}} \right) \\
		& = O \left( g^3 \cdot \frac{4}{2^{2g-5} \cdot 4} \cdot \frac{2^{2g-2}}{4} \cdot \frac{B^4}{g^4} \right)
		= B^4 \cdot O\left( \frac{1}{g} \right)
		.
	\end{align*}
	
	For the set $\calH_\closedgeod \setminus (\calH_\anyloop \cup \calH_\homsc)$, we approach the volume estimate in two steps: For a translation surface in $\calH_\closedgeod \setminus (\calH_\anyloop \cup \calH_\homsc)$, choose the shortest closed geodesic.
	Let $\alpha$ be the shortest saddle connection on this closed geodesic.
	Then $\alpha$ is not a loop and we can collapse a zero along $\alpha$ into another zero such that we obtain a closed geodesic of length at most~$\sfrac{6B}{g}$. This new closed geodesic is either a loop, starting at a zero of order $2$, or it contains a saddle connection of length at most $\sfrac{3B}{g}$ from a zero of order $2$ to a zero of order $1$. Hence we obtain
	the following finite--to--$1$, locally measure-preserving map:
	\begin{equation*}
		\calH_{\closedgeod}\setminus (\calH_\anyloop \cup \calH_\homsc) \to M_{\closedgeod} \times \left( \calH_{\twoloop} \cup \calH_\twotoone \right) \times \mathbb{D}_{\sfrac{6B}{g}} 
	\end{equation*}
	Note that this map is well-defined on a subset of full measure: The ghost double of $\alpha$ would not be defined only if there is a saddle connection of shorter length with the same direction (which only happens for a set of measure zero) or the ghost double defines a saddle connection which is homologous to $\alpha$ (which would mean that we are in $\calH_\homsc$).
	Here, $M_\closedgeod$ is the combinatorial set of order $g^2$ that expresses the choices of zeros that are collapsed into each other.
	We can estimate the volume of $ \calH_{\twoloop} \cup \calH_\twotoone$ using Equations~\eqref{eq:volume_loops_order-2} and \eqref{eq:volume_order-2} as
	\begin{equation*}
		\frac{ \Vol\left( \calH_{\twoloop} \cup \calH_\twotoone \right) }{\Vol(\calH_g(2,1,\ldots,1))}
		= O\left( \frac{B^2}{g^2} \right) + O\left( g \cdot \frac{B^2}{g^2} \right) 
		= B^2 \cdot O\left( \frac{1}{g} \right) 
		.
	\end{equation*}
	Therefore, we obtain with Equation~\eqref{eq:volume_stratum}
	\begin{equation*}
		\frac{ \Vol (\calH_\closedgeod \setminus (\calH_\anyloop \cup \calH_\homsc)) }{ \Vol (\calH_g) }
		= O(1) \cdot |M_\closedgeod|
		\cdot \frac{ \Vol \left( \calH_{\twoloop} \cup \calH_\twotoone \right) }{ \Vol (\calH_g) }
		\cdot \Vol \left( \mathbb{D}_{\sfrac{6B}{g}} \right)
		= B^4 \cdot O\left( \frac{1}{g} \right).
	\end{equation*}
	
	Finally, with Equations~\eqref{eq:volume_loop} and~\eqref{eq:volume_hom}, we have
	\begin{align*}
		\frac{\Vol(\calH_{\rm NG})}{\Vol(\calH_g)}
		& \leq \frac{ \Vol(\calH_\anyloop )}{\Vol( \calH_g)}
		+ \frac{ \Vol (\calH_{\cherry} \setminus (\calH_\anyloop \cup \calH_\closedgeod ))}{ \Vol (\calH_g) }
		+ \frac{ \Vol (\calH_\closedgeod \setminus (\calH_\anyloop \cup \calH_\homsc)) }{ \Vol (\calH_g) }
		+  \frac{ \Vol ( \calH_\homsc) }{ \Vol (\calH_g) } \\
		& = B^2 \cdot O\left( \frac1g \right)
		+ B^4 \cdot O\left( \frac{1}{g} \right) 
		+ B^4 \cdot O\left( \frac{1}{g} \right) 
		+ B^2 \cdot O\left( \frac{1}{g^2} \right) \\
		& = \max\{B^2, B^4\} \cdot O\left( \frac1g \right)
		. \qedhere
	\end{align*}
\end{proof}

We will also need a variation of \cref{Prop:generic-is-generic} for strata different from the principal stratum, specifically for the strata to which translation surfaces belong after collapsing along $K$ saddle connections.

\begin{definition}
	Fix $B \in \mathbb{R}_+$ as well as $K \in \NN, K <g$ and let $\calH' = \calH_g(2, \dots, 2, 1, \dots, 1)$ be the stratum with~$K$ zeros of order $2$ and $2g-2-2K$ zeros of order $1$.
	
	We define the following subsets of $\calH'$:
	\begin{align*} 
		\calH'_{\oneloop} & \coloneqq
		\left\{ \begin{aligned}
		(X,\omega) \in \calH' : \, & X  \text{ contains a loop of length at most  $\sfrac{6B}{g}$} \\
		& \qquad \text{starting at a zero of order } 1
		\end{aligned} \right\} \\
		\calH'_\cherry &\coloneqq 
		\left\{ \begin{aligned}
			(X,\omega) \in \calH' : \, & X \text{ contains a saddle connection of length at most $\sfrac{B}{g}$} \\
			& \qquad \text{and a saddle connection of length at most $\sfrac{2B}{g}$}\\
			& \qquad \text{that share a zero where all zeros are of order $1$}
		\end{aligned}
		\right\} \\
		\calH'_{\twoloop} & \coloneqq
		\left\{ \begin{aligned}
			(X,\omega) \in \calH' : \, & X  \text{ contains a loop of length at most  $\sfrac{2B}{g}$} \\
			& \qquad \text{starting at a zero of order } 2
		\end{aligned} \right\} \\
		\calH'_\twotostar &\coloneqq 
		\left\{ \begin{aligned}
			(X,\omega) \in \calH' : \, & X \text{ contains a saddle connection of length at most $\sfrac{2B}{g}$ } \\
			& \qquad \text{from a zero of order $2$ to another zero}
		\end{aligned}
		\right\}
	\end{align*}
	
	Similarly to before, we define the \emph{non-generic subset} of $\calH'$ to be
	\begin{equation*}
		\calH'_{\rm NG} = \calH'_\oneloop \cup \calH'_{\cherry} \cup \calH'_\twoloop \cup \calH'_\twotostar
	\end{equation*}
	and the \emph{generic subset} to be the complement of the non-generic subset: $\calH_{\rm Gen}' = \calH' \setminus \calH'_{\rm NG}$. 
\end{definition}

We can prove as in \cref{Prop:generic-is-generic} that  the volume of $\calH'_{\rm NG}$ is small:

\begin{proposition} \label{prop:generic-prime-is-generic}
	For fixed $K\in \NN$, we have
	\begin{equation*}
		\frac{\Vol(\calH_{\rm Gen}' )}{\Vol(\calH')} \to 1 \qquad\text{as}\qquad g \to \infty. 
	\end{equation*}
	More specifically, 
	\[
	\frac{\Vol(\calH_{\rm NG}') }{\Vol(\calH')}
	= O\left( \frac{K}{g} \right)
	\]
	where the implied constant does not depend on $K$ or $g$.
\end{proposition}

\begin{proof}
	Similar to  Equations~\eqref{eq:volume_loops_order-2} and \eqref{eq:volume_order-2}, we can use Equations~\eqref{eq:loop} and~\eqref{Eq:sc} to obtain
	\begin{equation*}
		\frac{\Vol(\calH'_\twoloop \cup \calH'_\twotostar ) }{\Vol(\calH')}
		=	9B^2 \cdot O\left( K \cdot \frac{1}{g^2} \right) 
		+ 6B^2 \cdot O\left( K g \cdot \frac{1}{g^2} \right)
		+ 9B^2 \cdot O\left( K^2 \cdot \frac{1}{g^2} \right)
		,
	\end{equation*}
	where the third (and first) term is dominated by the second as $K< g$.
	
	Furthermore, as in the proof of \cref{Prop:generic-is-generic}, we can show that
	\begin{equation*}
		\frac{\Vol(\calH'_{\cherry} \cup \calH'_\oneloop )}{\Vol(\calH')} = B^4 \cdot K \cdot O \left( \frac1g \right) + B^2 \cdot O\left( \frac{1}{g^2}\right)
		.
	\end{equation*}
	The additional factor of $K$ comes from the fact that we now have $K+1$ zeros of order~$2$ in~${\calH'}_{\twoloop} \cup {\calH'}_\twotostar$ instead of only one in $\calH_{\twoloop} \cup \calH_\twotoone$.
	
	In summary, we can deduce that 
	\begin{equation*}
		\frac{\Vol(\calH'_{\rm NG})}{\Vol(\calH')} = \max\{B^2, B^4\} \cdot K \cdot O \left( \frac1g \right) 
		\qquad\text{and}\qquad
		\frac{\Vol(\calH_{\rm Gen}' )}{\Vol(\calH')} \to 1. 
		\qedhere
	\end{equation*} 
\end{proof}

\section{Simultaneous collapsing} \label{sec:collapsing_simultaneously}

Fix now $0 < a_1 < b_1 \leq a_2 <b_2 \leq \dots \leq a_k <b_k \in \RR_+$ and choose $B \coloneqq b_k$.
Furthermore, fix $r_1,\dots, r_k \in \NN$ and define $K \coloneqq \sum_{i=1}^k r_i$.
To apply \cref{Thm:moments}, we have to determine 
\begin{equation*}
	\int_{\Hgen} (N_{g,[a_1,b_1]})_{r_1} \cdot \ldots \cdot (N_{g,[a_k,b_k]})_{r_k}
	.
\end{equation*}
Instead of calculating this integral directly, we replace $\Hgen$ by a feasible cover $\widehat \calH$, such that we can replace counting saddle connections on elements of $\Hgen$ by calculating the volume of $\widehat \calH$.

For this, let $\widehat \calH$ be the set of tuples~$(X, \omega, \Gamma)$ where 
$(X, \omega) \in \Hgen$, $\Gamma=(\Gamma_1, \dots, \Gamma_k)$  and 
$\Gamma_i=(\alpha_{i,1}, \dots, \alpha_{i, r_i})$
is an ordered list of disjoint oriented saddle connections with lengths in the interval $[\sfrac{a_i}g,  \sfrac{b_i}g]$ for every $i\in \{1,\ldots,k\}$. 
We sometimes think of  $\Gamma$ as a set and refer to a saddle connection~$\alpha \in \Gamma$,
which means $\alpha$ belongs to some~$\Gamma_i \in \Gamma$. 

Note that none of the saddle connections in $\Gamma$ is a loop and that no two saddle connections in $\Gamma$ or their ghost doubles intersect each other:
If the latter happened, there would exist another saddle connection of length at most $\sfrac{2B}{g}$ that shares a zero with one of the saddle connections in~$\Gamma$, so $(X, \omega)$ would be in $\calH_{\rm NG}$.
Therefore, we can simultaneously collapse all pairs of zeros that are connected by saddle connections in $\Gamma$ to obtain a~map
\[
\Phi \from \widehat \calH \to M_{r_1, \dots, r_k} \times \calH' \times \prod_{i=1}^k A_{a_i, b_i}^{r_i}
 . 
\]
Here $\calH' = \calH_g(2, \dots, 2, 1, \dots, 1)$ is the stratum of translation surfaces of genus $g$ with~$K$ zeros of order $2$ and the rest of order $1$. The set $M_{r_1, \dots, r_k}$ consists of the possible combinatorial data needed to record the labels of the (ordered) pairs of zeros of order $1$ in~$\widehat \calH$ which give rise to zeros of order $2$ in $\calH'$.
Finally, the set $A_{a_i, b_i} \subseteq \RR^2$ is an annulus in~$\RR^2$ centred at the origin with inner radius $\sfrac{a_i}{g}$ and outer radius $\sfrac{b_i}{g}$ and the notation $A_{a_i, b_i}^{r_i}$ means that we take~$r_i$ copies of this annulus.

 With the same arguments as for the map in~\eqref{eq:measure-preserving-map}, the map $\Phi$ is $3^K$--to--$1$ and locally measure-preserving.
We use this now to estimate the volume of $\widehat{\calH}$.

\begin{lemma} \label{lem:image-phi}
	For $\widehat{\calH}$ as above, we have
	\begin{equation*}
		\Vol( \widehat{\calH}) = 3^K \cdot \Vol( \calH_{\rm Gen}' ) \cdot \left( 1 + K \cdot O \left( \frac1g \right)\right) \cdot \prod_{i=1}^k  \left( \frac{ \pi(b_i^2-a_i^2)}{g^2}\right)^{r_i} \cdot |M_{r_1, \ldots, r_k}| 
	\end{equation*}
	where the implied constant does not depend on $K$ or $g$.

\begin{proof}
	For any $(X', \omega') \in \calH_{\rm Gen}' $, any element of $M_{r_1, \dots, r_k}$ and any $K$--tuple of vectors in 
	$\prod_{i=1}^k A_{a_i, b_i}^{r_i}$, 
	the $K$ zeros of order $2$ can be opened up (in $3^K$ different ways) in the direction of vectors in the chosen tuple and 
	labelled according to the element of $M_{r_1, \dots, r_k}$ to obtain~$(X, \omega)$ in~$\Hgen$.
	Furthermore, by choosing the oriented saddle connections in $X$ obtained from opening up zeros of order~$2$ as the elements of the tuples~$\Gamma_i$, we obtain a point in $\widehat \calH$. That is, the chosen point $(X', \omega')$ together with the further data is in the image of $\Phi$. 
	
	Therefore, we have
	\begin{equation*}
		\calH_{\rm Gen}' \times \prod A_{a_i, b_i}^{r_i} \times M_{r_1, \dots, r_k}
		\subseteq \Phi( \widehat \calH)
		\subseteq \calH' \times \prod A_{a_i, b_i}^{r_i} \times M_{r_1, \dots, r_k}
		.
	\end{equation*}
	
	As the map $\Phi$ is $3^K$--to--$1$, locally measure-preserving and nearly onto as just described and with the estimates on the volume of the complement of $\calH_{\rm Gen}'$ from \cref{prop:generic-prime-is-generic}, we obtain the claimed volume estimate.
\end{proof}
\end{lemma}

\section{The Length Spectrum} \label{sec:length_spectrum}
In this section, we check the assumptions of \cref{Thm:moments} for the set $\Hgen$. 
As before, fix $0 < a_1 < b_1 \leq a_2 <b_2 \leq \dots \leq a_k <b_k \in \RR_+$. 
For positive integers $r_1, \dots , r_k \in \NN$, define 
\[
Y_{g,r_1,...,r_k} := (N_{g,[a_1,b_1]})_{r_1} \dots (N_{g,[a_k,b_k]})_{r_k} \from \calH_g \to \NN_0. 
\]
Then $Y_{g,r_1,...,r_k}$ counts the number of (ordered) lists of length $k$ where the $i$--th item is an ordered 
$r_i$--tuple of saddle connections with lengths in $\left[\sfrac{a_i}g, \sfrac{b_i}g \right]$ on a surface  $(X, \omega) \in \calH_g$. 

Define
\[
\EE_{\rm Gen}(\param) = \frac{1}{\Vol(\Hgen)} \int_{\Hgen} \param. 
\]

\begin{proposition} \label{Prop:E}
For fixed $0 < a_1 < b_1 \leq \dots \leq a_k <b_k \in \RR_+$, we have
\[
\lim_{g \to \infty} \EE_{\rm Gen}(Y_{g,r_1,\dots,r_k})= \prod_{i=1}^k \lambda_{[a_i, b_i]}^{r_i}
\qquad \text{with} \qquad
\lambda_{[a_i, b_i]}= 8 \pi(b_i^2-a_i^2)
.
\]
\end{proposition}

\begin{proof}
Let $K \coloneqq \sum_{i=1}^k r_i$ as before.
For every $(X,\omega) \in \Hgen$, every suitable set of $K$ saddle connections gives rise to $2^K$ elements in $\widehat \calH$ with different choices of orientations for the saddle connections in $\Gamma$. Hence we have 
\[
\EE_{\rm Gen}(Y_{g,r_1,\dots,r_k}) = \frac 1{\Vol{\Hgen}} \cdot \frac{1}{2^K} \int_{\widehat \calH} \One. 
\]

The volume estimate from \cref{lem:image-phi} implies
\begin{equation*}
	\int_{\widehat \calH} \One = 3^K \cdot \Vol(\calH_{\rm Gen}')
	\cdot \prod_{i=1}^k  \left( \frac{\pi(b_i^2-a_i^2)}{g^2}\right)^{r_i}
	\cdot |M_{r_1, \ldots, r_k}| 
	\cdot \left( 1 + K \cdot O \left(\frac1g \right) \right)
\end{equation*}
with the notation as in \cref{sec:collapsing_simultaneously}.

An element in $M_{r_1, \ldots, r_k}$ is a way to choose $2K$ zeros (with labels) successively out of a set of $2g-2$ zeros with labels, hence
\[
|M_{r_1, \dots, r_k}| 
= \frac{ (2g-2)! }{(2g-2-2K)!}
.
\]
Combining these formulas, we get
\begin{equation*}
	\EE_{\rm Gen} (Y_{g,r_1,\dots,r_k}) 
	 =  \frac{\Vol(\calH_{\rm Gen}')}{\Vol(\Hgen)}
	\cdot \frac{3^K}{2^K}
	\cdot \frac{ (2g-2)! }{ (2g-2-2K)!} 
	\cdot \prod_{i=1}^k  \left(  \frac{\pi(b_i^2-a_i^2)}{g^2}\right)^{r_i} 
	\cdot \left( 1 + K \cdot O \left( \frac1g \right) \right)
	.
\end{equation*}
Furthermore, we have from \cref{Prop:generic-is-generic,prop:generic-prime-is-generic} and Equation~\eqref{eq:volume_stratum} that
\begin{equation*}
	\lim_{g \to \infty } \frac{\Vol(\calH_{\rm Gen}')}{\Vol(\Hgen)}
	= \lim_{g \to \infty} \frac{\Vol(\calH')}{\Vol(\calH_g)}
	= \frac{4}{3^K \cdot 2^{2g-2-2K}} \cdot \frac{2^{2g-2}}{4}
	= \frac{2^{2K}}{3^K} = \left( \frac43\right)^K. 
\end{equation*}
Taking the limit of $\EE_{\rm Gen} (Y_{g,r_1,\dots,r_k})$ for $g\to\infty$ and using $K = \sum_{i=1}^k r_i$ yields
\begin{align*}
	\lim_{g \to \infty} \EE_{\rm Gen} (Y_{g,r_1,\dots,r_k}) 
	& = \lim_{g \to \infty} \left( \frac43\right)^K
	\cdot \frac{3^K}{2^K}
	\cdot  (2g)^{2K} 
	\cdot \frac{1}{g^{2K}}
	\cdot \prod_{i=1}^k  \left(  \pi(b_i^2-a_i^2) \right)^{r_i} \\
	& = \prod_{i=1}^k  \left(  8 \pi(b_i^2-a_i^2) \right)^{r_i}
	. \qedhere
\end{align*}
\end{proof}

\section{ From the generic subset to the whole stratum} \label{sec:finish}

Let $\PP(\param)$ denote the probability of an event in $\calH_g$ and let $\PP_{\rm Gen}(\param)$ denote the probability of an event in $\Hgen$. For a random variable $N \from \calH_g \to \NN_0$ and $k \in \NN$, we have 
\[
\PP_{\rm Gen} (N=k) =  \frac{\Vol \Big\{ (X, \omega) \in \Hgen \ST N(X, \omega) =k \Big\} }{ \Vol(\Hgen) }
. 
\]
To compare $\PP(\param)$ and $\PP_{\rm Gen}(\param)$, note that
\begin{align*} 
	\Vol \Big\{ (X, \omega) \in \calH_g \ST N(X, \omega) =k \Big\} - \Vol(\Hng)
	& \leq \Vol \Big\{ (X, \omega) \in \Hgen \ST N(X, \omega) =k \Big\} \\
	& \leq  \Vol \Big\{ (X, \omega) \in \calH_g \ST N(X, \omega) =k \Big\}
	.
\end{align*} 
From \cref{Prop:generic-is-generic}, we know that 
\[
\frac{\Vol(\Hgen)}{ \Vol(\calH_g) } \to 1 
\qquad\text{and}\qquad \frac{\Vol(\Hng)}{ \Vol(\calH_g) } \to 0
\qquad\text{as}\qquad g \to \infty.
\]
Hence, if $\lim_{g \to \infty} \PP_{\rm Gen} (N=k)$ exists, so does $\lim_{g \to \infty} \PP (N=k)$ and the limits
are equal. The same holds for any event involving multiple random variables. 

In the case discussed here,  from \cref{Prop:E} and \cref{Thm:moments}, it follows that 
\[
\lim_{g \to \infty} \PP_{\rm Gen}
\left( N_{g,[a_1,b_1]} = n_1,..., N_{g,[a_k,b_k]} = n_k \right) = 
\prod_{i=1}^k \frac{\lambda_{[a_i, b_i]}^{n_i} e^{-\lambda_{[a_i,b_i]}}}{n_i!}
.
\]
Therefore, 
\[
\lim_{g \to \infty} \PP
\left( N_{g,[a_1,b_1]} = n_1,..., N_{g,[a_k,b_k]} = n_k \right) = 
\prod_{i=1}^k \frac{\lambda_{[a_i, b_i]}^{n_i} e^{-\lambda_{[a_i,b_i]}}}{n_i!}
.
\]
This finishes the proof of \cref{Thm:Main}.

\section{Proof of \texorpdfstring{\cref{thm:slow-growing_exceptions,thm:order_m}}{the variations of the main theorem}} \label{sec:variations}

In this section, we outline how to prove the two versions of the main theorem for other strata than the principal stratum.

We start with $\calH_g(m,\ldots,m)$ where $m$ is a fixed number that divides $2g-2$.
To adapt the proof of \cref{Prop:generic-is-generic} to this situation, we replace all the zeros of order~$1$ by zeros of order $m$. The collapsing procedure still works when considering all of the potential~$m$ ghost doubles simultaneously. As $m$ is fixed, the values of the Siegel--Veech constants and the combinatorial constants are changed but not the order of $g$ in these terms. Hence, the statement of \cref{Prop:generic-is-generic} and the analogous version of \cref{prop:generic-prime-is-generic} is also true for~$\calH_g(m, \ldots, m)$.

The map $\Phi$ from \cref{sec:collapsing_simultaneously} is replaced by the map
\begin{equation*}
	\widehat \calH \to M_{r_1, \dots, r_k} \times \calH'' \times \prod_{i=1}^k A_{a_i, b_i}^{r_i}
	,
\end{equation*}
where $\calH'' = \calH_g(2m,\ldots,2m,m,\ldots,m)$ is the stratum of translation surfaces of genus $g$ with~$K$ zeros of order $2m$ and the remaining zeros of order $m$. Note that this map is $(2m+1)^K$--to--$1$ which changes the volume of $\widehat{\calH}$ to
\begin{equation*}
	\Vol( \widehat{\calH}) = (2m+1)^K \cdot \Vol( \calH_{\rm Gen}'' ) \cdot \left( 1 + K \cdot O \left( \frac1g \right)\right) \cdot \prod_{i=1}^k  \left( \frac{\pi(b_i^2-a_i^2)}{g^2}\right)^{r_i} \cdot |M_{r_1, \ldots, r_k}| 
	.
\end{equation*}
With this, we can now do the same calculation as in \cref{Prop:E}. In particular, we have
\begin{equation*}
|M_{r_1, \dots, r_k}| 
= \frac{ \left( \frac{2g-2}{m} \right)! }{\left( \frac{2g-2}{m}-2K \right)!}
\end{equation*}
and
\begin{equation*}
	\lim_{g \to \infty } \frac{\Vol(\calH_{\rm Gen}'')}{\Vol(\Hgen)}
	= \frac{4}{(2m+1)^K \cdot (m+1)^{\frac{2g-2}{m}-2K}} \cdot \frac{(m+1)^{\frac{2g-2}{m}}}{4}
	= \left( \frac{(m+1)^2}{2m+1}\right)^K. 
\end{equation*}
Hence with $K = \sum_{i=1}^k r_i$, we obtain
\begin{align*}
	\lim_{g \to \infty} \EE_{\rm Gen} (Y_{g,r_1,\dots,r_k}) 
	& = \! \lim_{g \to \infty} \! \left( \frac{(m+1)^2}{2m+1} \right)^K
	\!\!\! \cdot \frac{(2m+1)^K}{2^K}
	\! \cdot \! \left( \frac{2g-2}{m} \right)^{2K} 
	\!\!\! \cdot \! \frac{1}{g^{2K}}
	\cdot \! \prod_{i=1}^k  \left(  \pi(b_i^2-a_i^2) \right)^{r_i} \\
	& = \prod_{i=1}^k  \left( \left( \frac{m+1}{m} \right)^2 2\pi(b_i^2-a_i^2) \right)^{r_i}
\end{align*}
which proves \cref{thm:order_m}.

For $\calH_g(m_1, m_2, \dots, m_{\ell(g)}, 1, \dots, 1)$ where $m_i \leq f(g)$, $i=1, \dots, \ell(g)$, we consider the following:
We use Equation~\eqref{Eq:sc} to show that the volume proportion of translation surfaces with a saddle connection of length at most $\sfrac{B}{g}$ from a non-simple zero to any other zero is at most of the order of
\begin{equation*}
	(f(g)+1)^2 \cdot \ell(g) \cdot g \cdot \left( \frac{B}{g} \right)^2 = O \left( \frac{f(g)^2 \cdot \ell(g)}{g} \right)
	.
\end{equation*}
Therefore, if $\frac{f(g)^2 \cdot \ell(g)}{g} \to 0$ for $g\to \infty$, the occurrence of these saddle connections is dominated by the occurrence of saddle connections from a simple zero to another simple zero that are used in the volume calculation in \cref{Prop:E}. Hence, the same reasoning as in the principal stratum works to obtain \cref{thm:slow-growing_exceptions}.

\bibliographystyle{alpha}
\bibliography{literature}

\end{document}